%% file: main.tex
\documentclass{article}
\usepackage[utf8]{inputenc}
\usepackage{amsmath,amsthm,amssymb}
\usepackage{fullpage}
\usepackage[pagebackref]{hyperref}
\usepackage[dvipsnames]{xcolor}
\usepackage{graphics}
\usepackage{epsfig}

\usepackage{nicematrix}

\hypersetup{
colorlinks=true,
urlcolor=blue,
linkcolor=blue,
citecolor=OliveGreen,
}

\def\hpic #1 #2 {\mbox{$\begin{array}[c]{l} \epsfig{file=#1,height=#2} \end{array}$}}
\def\vpic #1 #2 {\mbox{$\begin{array}[c]{l} \epsfig{file=#1,width=#2} \end{array}$}}

\newtheorem{theorem}{Theorem}[section]
\newtheorem*{theorem*}{Theorem}

\newtheorem*{proposition*}{Proposition}
\newtheorem{lemma}[theorem]{Lemma}
\newtheorem*{lemma*}{Lemma}
\newtheorem{corollary}[theorem]{Corollary}
\newtheorem*{conjecture*}{Conjecture}

\newtheorem*{fact*}{Fact}

\newtheorem*{hypothesis*}{Hypothesis}

\theoremstyle{definition}
\newtheorem{definition}[theorem]{Definition}

\newtheorem*{claim*}{Claim}

\newtheorem{remark}[theorem]{Remark}
\newtheorem*{remark*}{Remark}

\newtheorem*{observation*}{Observation}

\title{On Eigenvalue Gaps of Integer Matrices}
\author{Aaron Abrams\footnote{Washington and Lee University, {\tt abramsa@wlu.edu}.}, Zeph Landau\footnote{UC Berkeley, {\tt zlandau@berkeley.edu}.}, Jamie Pommersheim\footnote{Reed College, {\tt jamie@reed.edu}.}, Nikhil Srivastava\footnote{UC Berkeley. Supported by NSF Grant CCF-2009011, {\tt nikhil@math.berkeley.edu}.}}
\date{\today}

\newcommand{\C}{{\mathbb C}}

\newcommand{\Z}{{\mathbb Z}}
\newcommand{\B}{{\mathcal B}}
\renewcommand{\P}{{\mathcal P}}
\newcommand{\M}{{\mathcal M}}
\DeclareMathOperator{\Tr}{Trace}
\DeclareMathOperator{\spec}{{\mathfrak{spec}}}

\begin{document}
\maketitle

\begin{abstract}
    Given an $n\times n$ matrix with integer entries in the range $[-h,h]$, how close can two of its distinct eigenvalues be?
    
The best previously known examples (\cite{parlett1992minimum,wilkinson1988algebraic}) have a minimum gap of $h^{-O(n)}$.
Here we give an explicit construction of matrices with entries in $[0,h]$ with two eigenvalues separated by at most $h^{-n^2/16+o(n^2)}$.  Up to a constant in the exponent, this agrees with the known lower bound of $\Omega((2\sqrt{n})^{-n^2}h^{-n^2})$ \cite{mahler1964inequality}. Bounds on the minimum gap are relevant to the worst case analysis of algorithms for diagonalization and computing canonical forms of integer matrices (e.g. \cite{dey2021bit}).

In addition to our explicit construction, we show there are many matrices with a slightly larger gap of roughly $h^{-n^2/32}$.  We also construct 0-1 matrices which have two eigenvalues separated by at most $2^{-n^2/64+o(n^2)}$.


\end{abstract}

\section{Introduction}

We study the following question:
\begin{quote}
    Given an $n\times n$ matrix with integer entries in the range $[-h,h]$, how close can two of its distinct eigenvalues be?
\end{quote}
Let $g(n,h)$ be this minimum. We will be concerned with the asymptotic behavior of $g$ for large $n$ and $h$. 

Besides being of intrinsic interest, the question of eigenvalue gaps of integer matrices bears on the computational efficiency of certain matrix algorithms, as we discuss in Section \ref{sec:motivation} below.

Already in 1959, Olga Taussky \cite{taussky} began an AMS invited address on the broad theory of integer matrices by saying, ``This subject is very vast and very old.''  
Considering this, and of course the central role of eigenvalues in the subject, it is perhaps surprising that even today the gap problem has received relatively little attention.  Previous upper and lower bounds for $g(n,h)$ are miles apart. Upper bounds, given by explicit construction, are close to $h^{-n}$, while the best known lower bounds are less than  $h^{-n^2}$. 

In this note we improve the upper bound to 
\begin{equation}\label{result}
g(n,h) \le h^{-n^2/16+o(n^2)}.
\end{equation}
By analyzing a particularly nice family of matrices, we are able to find explicit examples of matrices exhibiting this gap.  These matrices are nearly irreducible, in a sense we will describe.  In addition, we show that the number of distinct eigenvalues of these matrices is large enough to imply existence of many matrices with a slightly larger gap, roughly $h^{-n^2/32}$.  The matrices guaranteed by this argument are not irreducible, however.

We remark that the study of BOunded HEight Matrices of Integers, into which our results fall, has relatively recently been coined ``Bohemian'' matrix theory; see \cite{bohemian} for some of the history as well as new results and spectacular images of the distributions (though not the gaps) of eigenvalues of several related families of matrices.

\subsection{Background}

The best previously known upper bound on $g$ is
\begin{equation}\label{eqn:ub}
    g(n,h)<2h^{-(n-2)}.
\end{equation}
Building on an example of Wilkinson \cite[p. 308]{wilkinson1988algebraic}, Parlett and Lu \cite[Theorem 2.3]{parlett1992minimum} proved \eqref{eqn:ub} for the symmetric, tridiagonal matrix with ones above and below the diagonal and diagonal entries $h,0,0,\ldots,0,h$.  In particular all entries of this matrix are non-negative.


On the other hand, the best known lower bound on $g$ is 
\begin{equation}\label{eqn:lb} 
    g(n,h) \geq (2\sqrt{n}h)^{-n(n-1)}=\Omega((2\sqrt{n})^{-n^2}h^{-n^2}).
\end{equation} 
This follows from Mahler's result \cite{mahler1964inequality} that the minimum gap between distinct roots of a degree $d$ integer {\em polynomial} of height\footnote{We refer to the largest coefficient/entry of an integer polynomial/matrix as its height.} $H$ is at least $H^{-(d-1)}$, applied to the characteristic polynomial of an $n\times n$ matrix of height $h$.  Such a characteristic polynomial has coefficients crudely bounded by $H\le 2^n(\sqrt{n}h)^n$ via Hadamard's inequality.

To put the huge discrepancy between the bounds \eqref{eqn:ub} and \eqref{eqn:lb} into context, consider for comparison that the corresponding question regarding gaps between zeros of integer polynomials of height $H$ has been resolved up to a constant factor in the exponent of $H$; see \cite{mahler1964inequality,collins1974minimum,mignotte1982some,bugeaud2004distance,bugeaud2011root, gotze2016discriminant}. The sharpness of the bounds for polynomials does not immediately translate to correspondingly sharp bounds for matrices because the height blowup when passing from a matrix to its characteristic polynomial can vary exponentially (as studied e.g. in \cite{chan2020upper} for Hessenberg matrices).

In Section \ref{sec:explicit} of this paper we show that a family of polynomials constructed by Mignotte \cite{mignotte1982some} to have small gaps between real roots can in fact be realized as the characteristic polynomials of integer matrices of exponentially smaller height.  Thus the upper bound for polynomials can be transported to matrices after all.

\begin{theorem}\label{thm:explicit} 
    For every odd integer $n$ and $h\ge 2$, there is a $(2n+1)\times (2n+1)$ matrix $B$ with entries in $\{0,1,\ldots, h\}$ and two real eigenvalues that differ by at most $h^{-{\frac{(n+3)(n-3)}{4}}}$. The two eigenvalues are roots of an irreducible (over $\mathbb{Z}$) factor of $\chi(B)$ of degree $n+1$.
\end{theorem}

The bound \eqref{result} follows.
In the regime $h=\Omega(n^c)$ this matches the lower bound \eqref{eqn:lb} up to a constant factor in the exponent of $h$.

As 0-1 matrices are of significant independent interest, we observe that the preceding theorem can easily be modified to cover this case, with a similar result.

\begin{corollary}\label{cor:01} 
    For every odd integer $n$, there is a $(4n+2)\times (4n+2)$ matrix $B$ with entries in $\{0,1\}$ and two eigenvalues at distance at most $2^{-{\frac{(n+3)(n-3)}{4}}}$. The two eigenvalues are roots of an irreducible (over $\mathbb{Z}$) factor of $\chi(B)$ of degree $2n+2$.
\end{corollary}

\subsection{Motivation}\label{sec:motivation}

This work is motivated in part by connections between the eigenvalue gaps of integer matrices and the algorithmic problem of efficiently computing the eigenvalues, eigenvectors, and more generally the Jordan Normal Form of an integer matrix (see e.g. the introduction of \cite{dey2021bit} for a detailed discussion). In particular, the number of bits of precision with which rational arithmetic must be performed in order to resolve distinct eigenvalues of an $n\times n$ integer matrix with entries in $[-h,h]$ is roughly $\log g(n,h)$ in the worst case, and this appears as a multiplicative factor in the running time of the currently best known algorithms for these problems (in particular, the bound $\eqref{eqn:lb}$ is used to establish the running time of the algorithm in \cite{dey2021bit}). Thus, improvements on minimum eigenvalue gap bounds --- particularly in the exponent of $h$ --- translate directly into improved worst case complexity estimates for these problems in certain models of computation. Our results show that such asymptotic improvements are not possible in the generality of arbitrary integer matrices when $h$ is sufficiently large depending on $n$. 

The irreducibility property in Theorem \ref{thm:explicit} shows that closely spaced eigenvalues of integer matrices must be dealt with numerically in the worst case, as there is no integer similarity which separates them via block-diagonalization. Even the reducible examples in Theorem \ref{thm:many} (below) present a serious algorithmic difficulty, as any direct sum of two such matrices can be conjugated by a unimodular integer matrix of small height (such as a Hadamard matrix, or a random such matrix) to yield an integer matrix with small eigenvalue gap which may or may not be easy to block-diagonalize in a way which separates the eigenvalues.

\subsection{Approach}

The construction used to prove Theorem \ref{thm:explicit} is based on the following circle of ideas.  If $M$ is an $n\times n$ matrix with monic characteristic polynomial $\chi(M)=\sum\limits_{i=0}^n a_i t^i$ then the $a_i$ are symmetric functions of the eigenvalues of $M$.  The sum $\sigma_j$ of the $j$th powers of the eigenvalues is also a symmetric function of the eigenvalues, and in particular the collection $\{\sigma_j\}_{j=1}^n$ determines the numbers $a_i$.  For instance $\sigma_1$, the trace of $M$, is equal to $-a_{n-1}$.  If $M$ is the (weighted) adjacency matrix of a weighted, directed graph $G$, then up to a factor of $j$, $\sigma_j=\Tr(M^j)$ is the total weight of all cycles in $G$ of length $j$.  Therefore, to construct lots of matrices with distinct characteristic polynomials, we  construct lots of weighted, directed graphs with different cycle weights $\sigma_j$. All of our matrices have a Hessenberg structure which simplifies the relevant combinatorics.

It turns out that the Mignotte polynomials arise from our construction, which enables us to give the explicit examples of Theorem \ref{thm:explicit}.  In addition, we apply a number theoretic result to show that a significant number of our polynomials are irreducible over $\Z$.
As distinct irreducible integer polynomials do not share roots, this implies that there are in some sense as many distinct eigenvalues of bounded height matrices as one would expect.

\begin{theorem}\label{thm:many}
    Let $\spec(n,h)=\{\lambda \in \C : \lambda \mbox{ is an eigenvalue of an } n\times n \mbox{ integer matrix of height } h \}$.  
    If $n$ is a power of $2$ then this set has size 
    $$|\spec(2n+1,h)| \ge \frac{2n}{5^{2n}}\cdot h^{n^2}.$$
\end{theorem}


As elements of $\spec(n,h)$ are bounded in magnitude by $nh$,
Theorem \ref{thm:many} implies that there are many examples of integer matrices with small eigenvalue gaps: simply choose any two nearby eigenvalues in $\spec(n/2,h)$, and take the direct sum of the matrices from which these two eigenvalues arise. In contrast to Theorem \ref{thm:explicit}, the nearby eigenvalues in this case do not come from a single irreducible factor. The proof of Theorem \ref{thm:many} appears in Section \ref{sec:many}.

We note that the distribution of the zeros and discriminant of integer polynomials was studied in \cite{beresnevich2010distribution,beresnevich2016integral}, which established certain density results in that context.

We also remark that ``minimal height companion matrices'' were studied in  \cite{chan-corless}, and it is possible that our construction in Theorem \ref{thm:explicit} produces such companion matrices for the Mignotte polynomials.

\subsection*{Acknowledgement}
We thank Henry and Julie Landau for the many ways they made this work possible. 

\section{A Set of Matrices}
The following set of (lower Hessenberg) matrices is used in all of our results.
Let $\B$ denote the set of $(2n+1 )\times (2n+1)$ matrices $$B=(B_{i,j})_{i\in[-n,n], j\in [-n,n]}$$ 
of the  form:
\begin{itemize}
\item 
    $B_{i, i+1}  =
    \begin{cases}
      1 & \text{for } i \in [-n, 0]\\
      h & \text{for } i \in [1, n-1].
    \end{cases}      
    $
\item 
    $B_{ij} \in \{0,1,\dots,h-1\}$ for $i \in [1, n]$ and $j \in [-n,-1]$; we let $A$ denote this lower left $n\times n$ corner of $B$.
\item 
    All other entries of $B$ are zero.
\end{itemize}
See Figure \ref{fig:matrix}.  Note that $|\B|=h^{n^2}$.

\begin{figure}[h]
\begin{center}
\hpic {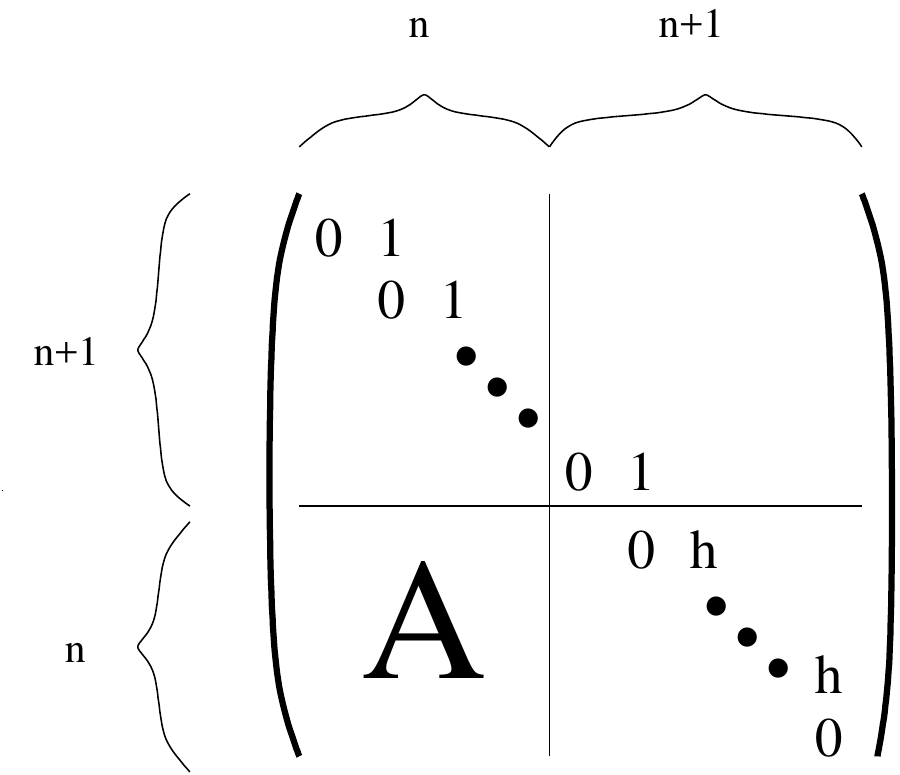} {2.5in}
\end{center}
\caption{A matrix in $\B$.}
\label{fig:matrix}
\end{figure}

\subsection{Graphical interpretation}
We can think of a matrix $B\in\B$ as the weighted adjacency matrix of a directed, weighted graph on $2n+1$ vertices.  From left to right we label the vertices with integers from $-n$ to $n$, with directed edges to the right joining consecutive vertices.  The weight of the edge $(i,i+1)$ is $1$ for $i\in [-n,0]$ and $h$ for $i \in [1,n-1]$, accounting for the superdiagonal entries of $B$. In addition, we can choose to add any subset of the additional $n^2$ edges that connect a vertex $i \in [1,n]$ to a vertex $j \in [-n,-1]$; any of these added edges can be given arbitrary weight from $\{1,\ldots,h-1\}$.  Each of these added edges corresponds to a nonzero entry in the matrix $A$.  Figure \ref{fig:graph} shows an example of this for $n=4$ and $h=2$ with the indicated matrix $A$.

\begin{figure}[h]
\begin{center}
    \hpic {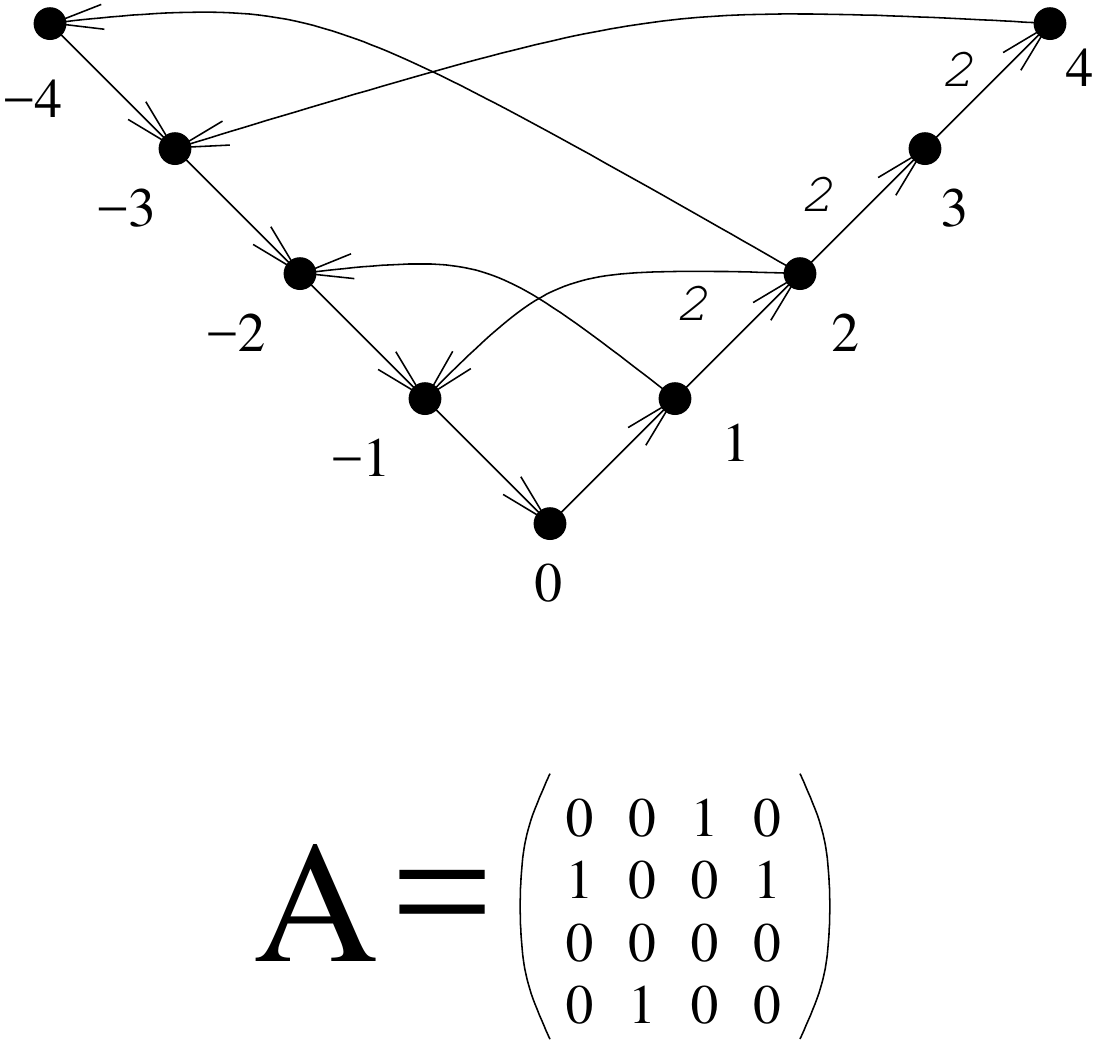} {2in}
\end{center}
\caption{A weighted digraph.}
\label{fig:graph}
\end{figure}

\section{The Characteristic Polynomials}
In this section we describe the set of characteristic polynomials of the matrices in $\B$.
Let $\P$ be the set of monic integer polynomials of degree $2n+1$ of the following form:

\begin{equation}\label{charpoly}
    t^{2n+1} - a_{2n-2}t^{2n-2} - a_{2n-3}t^{2n-3} - \cdots - a_1 t -a_0, 
\end{equation}
where the coefficients $a_i$ satisfy
\begin{itemize}
    \item $a_{2n}=a_{2n-1}=0$,
    \item $a_{2n-k} \in \{0,1,\ldots, h^{k-1} -1\}$, for $k \in [2, n+1]$,
    \item  $\frac{a_{n-k}}{h^{k-1}} \in \{0,1,\ldots,h^{n-k+1} -1\}$, for $k\in[2, n]$.
\end{itemize}

As with $\B$, we have $|\P|=h^{n^2}$.

\begin{theorem} \label{thm:bijection}
    The characteristic polynomial $\chi$ acts as a bijection from $\B$ to $\P$.
\end{theorem}

This is proved by the following computations.

\subsection{Computation of the Characteristic Polynomial}

\paragraph{Via Permutations.}

\medskip

Consider the matrix $C=B-tI$.  For any permutation $\sigma$ of $\{-n, \dots, n\}$, let $p_{\sigma}=(-1)^{\sigma}\prod C_{i,\sigma(i)}$ be the corresponding term in $\det(C)$. It suffices to consider terms so that $p_{\sigma}\neq 0$. 

Suppose $\sigma$ does not access the lower left quadrant of the matrix.  Then $\sigma$ is supported on an upper triangular matrix. So $\sigma$ is the identity permutation and $p_{\sigma}=-t^{2n+1}$. 

On the other hand, suppose that $\sigma$ does access the lower left part of the matrix, i.e., suppose that for some  $a\in\{1,\dots, n\}$ and $b\in\{-n,\dots, -1\}$, we have $\sigma(a)=b$. With the assumption that $p_{\sigma}\neq 0$, it follows that $\sigma$ is determined uniquely: $\sigma$ must be the permutation that maps $a$ to $b$ but keeps the remaining elements in order. That is $\sigma(i)=i$ for all $i\notin [b,a]$, while $\sigma(i)=i+1$ for all $i\in [b,a-1]$ and $\sigma(a)=b$.  Thus $\sigma$ is an $(a-b+1)$-cycle, so its sign is $(-1)^{a-b}$.  For such $\sigma$, one computes that
$$
p_{\sigma} = B_{a,b} h^{a-1} t^{2n+b-a}.
$$
It follows that for $k<2n+1$, the coefficient $a_k$ of $t^k$ in Equation \eqref{charpoly} is given by
\begin{equation}\label{coeffs}
a_k= \sum_{a-b=2n-k} B_{a,b} h^{a-1}.
\end{equation}

\noindent
\paragraph{Via Graphs.}

For these matrices, the coefficient $a_k$ is the total weight of all simple loops on the associated graph starting at $0$ and having length $2n+1 -k$.  To see this, we start by noting that for $p_{\sigma}=(-1)^{\sigma}\prod C_{i,\sigma(i)}$ to be a term contributing to $a_k$ the permutation $\sigma$ must have exactly $k$ fixed points (contributing the power of $t^k$).  Suppose such a $\sigma$ has a size $\ell$ cycle $(i_1, i_2, \dots i_{\ell})$ with $i_j \in [-n,n]$, then for $p_{\sigma}$ to be nonzero it must be that $B_{i_1, i_2}B_{i_2, i_3} \dots B_{i_{\ell-1}, i_{\ell}} B_{i_{\ell}, i_{1}}$ is nonzero.   This implies that the cycle $(i_1, i_2,\dots i_{\ell})$ constitutes a loop on the associated graph of $B$. Because the graph is directed, any cycle must move from left to right along the ``V'' portion of the graph and then circle back from the right side of the ``V'' to the left using an edge corresponding to a nonzero entry of the associated $A$ matrix.  All such loops pass through the sequence of nodes $(-1,0,1)$ and thus $\sigma$ must consist of a single simple loop if $p_{\sigma}$ is nonzero.   

Determining $b_k$ therefore consists of weighing the simple loops of length $2n+1 -k$ on the associated graph. Notice that the sign of the contribution is always $+1$ since the sign of a permutation of a simple loop of size $2n+1 -k$ is $(-1)^{2n  -k}$ and the coefficient picks up another factor of $(-1)^k$ from the $k$ fixed points that contribute a factor of $(-t)$ each. Since each loop passes through $0$ we imagine tracing it out by starting at $0$.  Then any loop starts by going to the right for some number $a$ of steps, then takes a directed edge from $a>0$ to some $b<0$ (whose weight is $B_{a,b}$), and then travels to the right along the negative edges back to $0$.  The length of such a path is  $a-b +1$ and the weight of the path is $B_{a,b}h^{a-1}$.

Therefore we again obtain \eqref{coeffs}.
\noindent
\paragraph{One more description.}
The above computations show that one can read off the coefficients $a_i$ from the matrix $A$ as follows:  pad $A$ with zeroes on the left to form the $n\times (2n-1)$ matrix $D=\begin{bmatrix} 0 & | & A \end{bmatrix}$.  If we call the list of numbers $D_{j,j+i}$ the ``$i$th diagonal'' of $D$, then $a_i$ is the integer whose base $h$ expansion is given by the $i$th diagonal of $D$. Theorem \ref{thm:bijection} immediately follows.

\section{Explicit matrices with gap $h^{-O(n^2)}$} \label{sec:explicit}

Mignotte [M82] showed that the degree $d$ polynomial 
$$
m_{d,a}(X) = X^d  - 2(aX-1)^2
$$
has two real roots separated by less than $a^{-\frac{d+2}{2}}$.   

In this section, we exhibit matrices $B$ whose characteristic polynomials $\chi_B(t)$ are Mignotte polynomials multiplied by powers of $t$, thereby proving Theorem \ref{thm:explicit}. 

\subsection*{The $h=2$ case.}

\begin{definition}
    Let $n$ be a positive odd integer, and let $m=2n+1$.  We define the $m\times m$ {\it Mignotte matrix} to be the matrix whose superdiagonal contains $n+1$ $1$'s followed by $n-1$ $2's$, and whose other entries are 0 except
    \begin{align*}
        M_{n+3,1}&=1 \\
        M_{\frac{3n+3}{2}, \frac{n+1}{2}}&=2 \\
        M_{2n,n}&=1
    \end{align*}
\end{definition}

We then have the following
\begin{theorem}
    Let $n$ be a positive odd integer.  Then the $(2n+1)\times (2n+1)$ Mignotte matrix $M$ has characteristic polynomial
    $$
    \chi_M(t) = t^{n-2} m_{n+3, 2^{(n-3)/2}}(-t).
    $$
    Hence there are two roots of $\chi_M(t)$ that are separated by at most  $2^{-(n+5)(n-3)/4}$.
\end{theorem}

\begin{proof}
Computing $\chi_M(t)$ follows the same procedure we used to compute $\chi_B$ in Section 2.  The separation in the roots of $\chi_M(t)$ is then less than or equal to the separation of the roots of $m_{n+3, 2^{(n-3)/2}}(-t)$, and we can apply Mignotte's bound. \end{proof}

\begin{remark} The matrix $M$ is not in $\B$, but the method of computing $\chi_M$ is the same.  The three exceptional entries $1, 2, 1$ are colinear in the matrix $M$.  If desired, we could achieve the same characteristic polynomial by replacing the exceptional $2$ with a $1$ in the position $(\frac{3n+5}{2},\frac{n+3}{2})$ immediately to its southeast. This matrix is in $\B$.\end{remark}

\subsection*{The $h>2$ case.}

Now fix $h>2$.  For a positive odd integer $n$, with  $m=2n+1$, define an $m\times m$ matrix $M_n(h)$ to have superdiagonal containing $n+1$ $1$'s followed by $n-1$ $h's$, and whose other entries are 0 except
\begin{align*}
    M_{n+2,2}&=2 \\
    M_{ \frac{3n+1}{2}, \frac{n+3}{2}}&=4 \\
    M_{2n-1,n+1}&=2
\end{align*}

We then have the following
\begin{theorem}
    Let $n$ be a positive odd integer.  Then the $(2n+1)\times (2n+1)$  matrix $M=M_n(h)$ has characteristic polynomial
    $$
    \chi_M(t) = t^{n} m_{n+1, h^{\frac{n-3}{2}}}(-t).
    $$
    Hence there are two roots of $\chi_M(t)$ that are separated by at most  $h^{-(n+3)(n-3)/4}$.
\end{theorem}

\subsection*{The $0$-$1$ case.}
We now note that Corollary \ref{cor:01} follows by taking the ``double cover'' of any matrix obtained from the $h=2$ case of Theorem \ref{thm:explicit}, for example the Mignotte matrices described above. Starting with a matrix $M$ all of whose entries are $0,1$, or $2$, construct a matrix $M^{\prime}$ with twice as many rows and columns by replacing each entry of $M$ with a $2\times 2$ matrix in the following manner: $0$ is replaced by the zero matrix; $1$ is replaced by the identity matrix; $2$ is replaced by the all-ones matrix.  The matrix $M'$ is the double cover of $M$.

We note that there is a straightforward way of viewing this construction as a directed graph. Theorem \ref{thm:explicit} provides a weighted directed graph $G$, in which every edge has weight $1$ or $2$. The double cover $G'$ of $G$ will be an (unweighted) directed graph, constructed as follows: For each vertex $v$ of $G$, there are two vertices $v_0$ and $v_1$ of $G'$. For each directed edge $(v,w)$ of $G$ of weight 1, $G'$ will have two directed edges of the form $(v_i, w_i)$; for each directed edge $(v,w)$ of $G$ of weight 2, $G'$ will have four directed edges of $G'$ the form $(v_i, w_j)$. The directed graph $G'$ obtained in this way, starting with the weighted directed graph shown in Figure \ref{fig:graph}, has adjacency matrix $M'$.

Given any eigenvector for $M$, there is a corresponding eigenvector for $M'$ with the same eigenvalue. As $M'$ is a $0$-$1$ matrix, Corollary \ref{cor:01} follows.

\section{A large collection of distinct eigenvalues}
\label{sec:many}

Here we prove Theorem \ref{thm:many}.  The idea is to find a lot of irreducible polynomials in $\P$.  These are provided by the following number theoretic lemma from \cite[Example 5.2]{mod5}.

\begin{lemma}\label{lem:mod5irr}
    If $q \equiv 1 \mod 4$ is prime and $a$ is a quadratic nonresidue in $F_q$, 
    then the polynomial $t^{2^k} - a$ is irreducible mod $q$.
\end{lemma}

Let $q=5$ and assume $h>1$ is not a multiple of $5$.
Let $p(t)=t^{2n}-a$ where $a$ is either $h^{n-2}$ or $2h^{n-2}$, whichever is a quadratic nonresidue mod 5.
By the above lemma, $p(t)$ is irreducible mod 5 (hence also over $\Z$) whenever $n$ is a power of $2$.


Now let $P\in\P$ be such that $P(t) \equiv tp(t)\mod 5$.
There are many such polynomials: one fifth of the possible values for each coefficient $a_i$ are acceptable.  By Theorem \ref{thm:bijection} there are corresponding matrices $B\in\B$ whose characteristic polynomials therefore have distinct irreducible factors each of degree $2n$.  The result is a collection of $\frac {2n}{5^{2n-1}} h^{n^2}$ distinct (complex) eigenvalues of elements of $\B$, proving Theorem \ref{thm:many}.

These eigenvalues are bounded in magnitude by $2hn$.
It follows that there are many pairs of numbers in $\spec(2n+1,h)$ that are within a distance (roughly) $h^{-n^2/2}$ of each other.  The block sum of two corresponding matrices is then a reducible $(4n+2)\times (4n+2)$ matrix with gap at most $h^{-n^2/2}$.

\begin{remark}
    Theorem \ref{thm:many} is an assertion about $(2n+1)\times (2n+1)$ matrices where $n$ is a power of $2$.  One can obtain a similar result in an arbitrary dimension $d$ by finding the largest $n=2^k$ such that $n\leq (d-1)/2$ and applying the above construction in dimension $2n+1$.  Taking a direct sum with an appropriately sized identity matrix will yield a collection of $d\times d$ matrices with the desired property that the union of their spectra is very large, nearly $h^{n^2}$.  In the unlucky case that $(d-1)/2$ is just larger than a power of $2$, we lose a factor of 2 in the dimension with $n\approx d/4$ and thus we get only about  $h^{d^2/16}$ distinct eigenvalues. 

Igor Shparlinski pointed out to us the following idea for improving this bound to about $h^{d^2/4}$ in arbitrary dimension $d$.  Instead of restricting the degrees of the irreducible polynomials $t^n-a$ to the values $n=2^k$ as in Lemma \ref{lem:mod5irr}, one can allow degrees of the form $n=2^k3^\ell$.  By careful choice of $a$, irreducibility can still be guaranteed e.g.~from Theorem 3.75 of \cite{lidl_niederreiter_1996}. The (increasing) sequence of integers of the form $2^k3^\ell$ has ratios approaching $1$ \cite{Tijdeman1974}, allowing one to apply our construction with $n\approx d/2$, improving the bound. 
\end{remark}

\section{Questions}

%

The matrices we construct in our proofs are not symmetric. The question of the correct asymptotics of $g(n,h)$ restricted to normal or symmetric matrices remains open, and it is entirely possible that the truth in that setting is closer to \eqref{eqn:ub}.

Our matrices also have non-negative entries.  It is possible that our bound can be improved for general height $h$ matrices by including negative entries.

In studying this problem we encountered several related questions of interest.  In what follows let $\M$ denote the set of $n\times n$ integer matrices of height $h$.  One could also add the adjectives ``normal,'' ``symmetric,'' or ``with non-negative entries'' if desired.

\begin{enumerate}
    \item How big is the set of eigenvalues of elements of $\M$?
    \item How big is the set of characteristic polynomials of elements of $\M$?
    \item Given $A\in\M$, how many elements of $\M$ can be similar to $A$?
    \item If $A,B\in\M$ are similar, can one bound the height of the smallest conjugating matrix?
    \item How small can one make the entries of a matrix with given (integer) characteristic polynomial?
    \item What is the smallest gap between roots of a real-rooted integer polynomial with bounded coefficients?
\end{enumerate}
\bibliographystyle{alpha}
\input{main.bbl}


\end{document}

%% file: main.bbl
\newcommand{\etalchar}[1]{$^{#1}$}